\documentclass[12pt]{amsart} 

\usepackage{amssymb}
\usepackage{amsmath}
\usepackage{amsthm}
\usepackage{verbatim}
\usepackage{url}
\usepackage{color}
\usepackage{mathabx}
\newcommand*{\todo}[1]{\textcolor{red}{#1}}

\usepackage{tikz}
\usepackage{pgfplots}
\usepackage{graphicx}
\usepackage{calc}
\newcommand{\op}{
  \mathop{
    \vphantom{\bigoplus} 
    \mathchoice
      {\vcenter{\hbox{\resizebox{\widthof{$\displaystyle\bigoplus$}}{!}{$\boxplus$}}}}
      {\vcenter{\hbox{\resizebox{\widthof{$\bigoplus$}}{!}{$\boxplus$}}}}
      {\vcenter{\hbox{\resizebox{\widthof{$\scriptstyle\oplus$}}{!}{$\boxplus$}}}}
      {\vcenter{\hbox{\resizebox{\widthof{$\scriptscriptstyle\oplus$}}{!}{$\boxplus$}}}}
  }\displaylimits 
}

\newcommand{\Pcal}{{\mathcal P}}

\newcommand{\KK}{{\mathcal K}}    

\newcommand{\CC}{\mathbb{C}}

\newcommand{\FF}{\mathbb{F}}

\renewcommand{\KK}{\mathbb{K}}
\newcommand{\PP}{\mathbb{P}}

\newcommand{\RR}{\mathbb{R}}
\renewcommand{\SS}{\mathbb{S}}
\newcommand{\WW}{\mathbb{W}}
\newcommand{\ZZ}{\mathbb{Z}}

\renewcommand{\root}{\xi}
\newcommand{\Root}{\xi}
\newcommand{\mcm}[3]{\newcommand{#1}[#2]{{\ensuremath{#3}}}} 
\mcm{\restric}{0}{\upharpoonright}

\numberwithin{equation}{section}

\newtheorem{theorem}[equation]{Theorem}
\newtheorem{lemma}[equation]{Lemma}

\newtheorem{question}[equation]{Question}
\newtheorem{corollary}[equation]{Corollary}

\newtheorem*{theorem*}{Theorem}

\theoremstyle{definition}
\newtheorem{defn}[equation]{Definition}
\newtheorem{example}[equation]{Example}

\theoremstyle{remark}
\newtheorem{remark}[equation]{Remark}

\begin{document}
\title{Extensions of Hyperfields}
\author{Steven Creech}
\email{screech6@math.gatech.edu}
\address{School of Mathematics,
         Georgia Institute of Technology, USA}

\date{\today}
\thanks{\textbf{Acknowledgement.} I would like to thank Professor Matthew Baker for coming up with this problem and advising me through its completion. }

\begin{abstract}

We develop a theory of extensions of hyperfields that generalizes the notion of field extensions. Since hyperfields have a multivalued addition, we must consider two kinds of extensions that we call weak hyperfield extensions and strong hyperfield extensions. For quotient hyperfields, we develop a method to construct strong hyperfield extensions that contain roots to any polynomial over the hyperfield. Furthermore, we give an example of a hyperfield that has two non-isomorphic minimal extensions containing a root to some polynomial. This shows that the process of adjoining a root to a hyperfield is not a well-defined operation.

\end{abstract}

\maketitle

\section{Introduction} \label{sec:intro}

Let $F$ be a field, and $f\in F[x]$ be an irreducible polynomial. Results from classical field theory tells us that we can always find some field extension $F\subseteq L$ such that $f$ contains a root in $L$. The construction of this field extension is quite simple, we just take the ring of polynomials over $F$ and quotient out be the ideal generated by $f$, that is $L=F[x]/\langle f\rangle$ is a field extension of $F$ that contains a root to $f$. Thus, given some field $F$, and a root $\root$ of some polynomial $f\in F[x]$, we can make sense of the field $F(\root)$ as the smallest field extension of $F$ containing the root $\root$. Furthermore, this extension is unique up to isomorphism. This minimal extension plays an important role in Galois theory.

\bigskip

In this paper, we set out to answer an analogous question over hyperfields. That is, given a polynomial $k$ over some hyperfield $K$, is there a hyperfield extension of $K$ that contains some root to $k$. Furthermore, if such extensions exist are the minimal extensions of $K$ containing a root unique up to isomorphism. That is for a hyperfield $K$ and a root $\Root$ to some polynomial $k$ can we join the root $\Root$ to $K$ and make sense of the hyperfield $K(\Root)$. As addition over hyperfields is a multivalued operation, there are two definitions one can take for a hyperfield extension namely we can require that the sum of two elements in the ground hyperfield is equal to sum in the extension. Alternatively, we could require set inclusion. When equality is held we shall say that we have a \textit{strong hyperfield extension}, and when inclusion is held we shall say it is a \textit{weak hyperfield extension}. 

\subsection{Statement of Results} In this paper, we will develop a method for constructing strong hyperfield extensions containing roots for quotient hyperfields. Furthermore, we shall exhibit an example of a hyperfield and a polynomial that contains two non-isomorphic minimal extensions. That is, for a hyperfield $K$ and a root $\Root$ to a polynomial $k$, adjoining the root to form $K(\Root)$ is not a well-defined operation. Without the uniqueness of minimal extensions, trying to develop a Galois theory over hyperfields becomes a challenge.

\subsection{Content Overview} In section \ref{sec:prelim} of this paper, we give the definitions of hyperfields, polynomials over hyperfields, and other related concepts. Furthermore, we shall review the construction of quotient hyperfields developed by Krasner in \cite{krasner1983class}. This construction will play a key role in our construction of extensions. In section \ref{sec:ext}, we construct strong hyperfield extensions for quotient hyperfields and prove that our construction is valid.
In section \ref{sec:minext}, we use the weak hyperfield $\WW$ and a polynomial with no roots over $\WW$ to construct two extensions containing roots. Then we show that one of the two extensions is minimal, but not contained in the other. This implies that there must be two non-isomorphic minimal extensions of $\WW$ containing a root.

\section{Preliminaries}\label{sec:prelim}

\subsection{Hyperfields}

Our main object of interest are hyperfields which were introduced by M. Krassner in 1956 \cite{krasner1956approximation}. Hyperfields are a generalization of fields with the difference being that addition in hyperfields is a multivalued operation. Thus, we can think of addition as a binary operator from $K\times K$ going to the nonempty subsets of $K$; formally, we can say that $\boxplus:K\times K\rightarrow \Pcal^\times(K)$. As addition is multivalued, we must make sense of associativity. We will first give the formal definition of a hyperfield and then give a discussion on associativity.

\begin{defn}
A \textit{hyperfield} is a $5$-tuple $(K,\boxplus,\odot,0,1)$ that satisfies the following conditions:

\begin{enumerate}

    \item $(K^\times=K\backslash\{0\},\odot, 1)$ forms an abelian group.
    
    \item For all $x\in K$ we have that $0\odot x=x\odot 0 = 0$
    
    \item The following distributive laws hold for all $x,y,z\in K$:
    
    \begin{enumerate}
    
    \item $x\odot(y\boxplus z)=(x\odot y)\boxplus (x\odot z)$
        
    \item $(x\boxplus y)\odot z=(x\odot z)\boxplus (y\odot z)$
        
    \end{enumerate}
    
    \item $(K,\boxplus, 0)$ forms a commutative hypergroup that is:
    
    \begin{enumerate}
    
        \item $\boxplus$ is an associative that is for all $x,y,z\in K$, we have that $(x\boxplus y)\boxplus z= x\boxplus(y\boxplus z)$
        
        \item $\boxplus$ is commutative that is for all $x,y\in K$, we have that $x\boxplus y= y\boxplus x$
        
        \item $0$ is an additive identity that is for all $x\in K$ we have that $0\boxplus x=x\boxplus 0=\{x\}$
        
        \item For all $x\in K$ there is a unique element $-x\in K$ such that $0\in x\boxplus -x=-x\boxplus x$
        
        \item $x\in y\boxplus z$ if and only if $-y \in -x \boxplus z$
        
    \end{enumerate}
    
\end{enumerate}

\end{defn}

{
Since $a\boxplus b$ is a set, we will need to define what it means to preform hyperaddition on a set $A$ and an element $c$, so we define:

\[
A\boxplus c=\bigcup_{x\in A} x\boxplus c
\]

We can further generalize this to say what happens when we preform hyperaddition on two sets $A$ and $B$:

\[
A\boxplus B=\bigcup_{x\in A, y\in B} x\boxplus y
\]

That is we can think of the hypersum of two sets as the union of all hypersums of pairs of elements from the sets.

These definitions allow us to recursively define arbitrarily long hypersums. That is we can recursively define $x_1\boxplus x_2\boxplus...\boxplus x_n$ as:
\[
x_1\boxplus x_2\boxplus...\boxplus x_n=\bigcup_{y\in x_1\boxplus x_2\boxplus...\boxplus x_{n-1}}y\boxplus x_n
\]
}


{
Now that we have defined hyperfield, let us give some examples.

\begin{example}
(Fields) A field $F$ forms a hyperfield with the hyperaddition as the original addition in the field that is $x\boxplus y=\{x+y\}$. The multiplication is just that of the field that is $x\odot y=x\cdot y$
\end{example}

\begin{example}
(Krasner Hyperfield) The Krasner hyperfield $\KK$ is a hyperfield over the set $\{0,1\}$ where $0$ acts as the additive identity and $1$ acts as the multiplicative identity. Thus, we have that $0\boxplus 1=1\boxplus 0=\{1\}$, and we have $1\boxplus 1=\{0,1\}$.

\end{example}

\begin{example}
(Hyperfield of Signs) The hyperfield of signs $\SS$ is a hyperfield over the set $\{0,1,-1\}$. The addition and multiplication are given by the following Cayley tables:

\begin{table}[h]
\parbox{.45\linewidth}{
\centering
\begin{tabular}{|l|l|l|l|}
\hline
$\boxplus$ & 0 & 1 & -1 \\ \hline
0 & \{0\} & \{1\} & \{-1\} \\ \hline
1 & \{1\} & \{1\} & \{0,1,-1\} \\ \hline
-1 & \{-1\} & \{0,1,-1\} & \{-1\} \\ \hline
\end{tabular}

\caption{Cayley table for addition for $\SS$}
}
\hfill

\parbox{.45\linewidth}{
\centering
\begin{tabular}{|l|l|l|}
\hline
$\odot$ & 1 & -1 \\ \hline
1 & 1 & -1 \\ \hline
-1 & -1 & 1 \\ \hline
\end{tabular}
\caption{Cayley table for multiplication for $\SS$}
}
\end{table}

The way to think of the hyperfield of signs is to interpret the element $1$ as any positive number, think of the element $-1$ as any negative number, and think of $0$ as $0$. 

So multiplication is defined in the obvious manner as $1\odot 1=1$ is interpreted as the product of two positive numbers is positive. The product $1\odot -1=-1$ is interpreted as the product of a positive and a negative is negative, and $-1\odot -1=1$ is interpreted as the product of two negative numbers is positive. 

We interpret addition as follows $1\boxplus 1=\{1\}$ we view this as the sum of two positive numbers is positive. We interpret $-1\boxplus -1=\{-1\}$ as the sum of two negative numbers is negative. The interesting case is the sum $1 \boxplus -1=\{0,1,-1\}$ which we interpret as the sum of a positive and a negative number can be either positive, negative, or zero. We will see later that this interpretation is quite natural due to the quotient hyperfield structure $\SS=\RR/\RR_{>0}$.

\end{example}

\begin{example}

(Weak Hyperfield of Signs) The weak hyperfield of signs or simply the weak hyperfield, $\WW$, is a hyperfield over the set $\{0,1,-1\}$ where multiplication is defined in the same manner as the hyperfield of signs, but addition is defined by $1\boxplus 1=\{1,-1\}$,\\ $-1\boxplus -1=\{1,-1\}$, and $1\boxplus -1=\{0,1,-1\}$. 

\end{example}
}

\subsection{Quotient Hyperfields}

Each of the examples we gave above are a special class of hyperfields called a quotient hyperfields. A quotient hyperfield is a hyperfield that is constructed from a field by quotienting out by a multiplicative subgroup. For a long time, it was an open question on whether or not all hyperfields admit a quotient structure; however, there exist non-quotient hyperfields see \cite{massouros1985methods}. As the construction of quotient hyperfields is important to understanding our construction of strong hyperfield extensions, we will go over the construction of quotient hyperfields; however, we will omit the proof that the quotient is a hyperfield see \cite{krasner1983class}.

Given a field $F$ and $G\leq F^\times$ a multiplicative subgroup of $F$, we construct the following equivalence relation on the elements of $F$. We say that $x\sim y$ if and only if there is some element $g\in G$ such that $x=gy$. We note that for $x\in F^\times$, $[x]$ is the coset of $F^\times/G$, so we can think of the equivalence classes as the multiplicative cosets where $[0]=\{0\}$.

We now form the quotient hyperfield $K=F/G$ by modding out by the equivalence $\sim$. The elements of $K$ are the equivalence classes under $\sim$, and we have the following induced operations:

\begin{enumerate}
    \item $[x]\odot [y]=[x\cdot y]$
    \item $[z]\in [x]\boxplus[y]$ if and only if there is some $z'\in[z]$ such that $z'=x'+y'$ for some $x'\in[x]$ and $y'\in[y]$
\end{enumerate}

The hyperaddition formulation might seem complicated; however, another way to think of the hypersum $[x]\boxplus [y]$ would be to think of $[x]$ and $[y]$ as cosets, then define \\$[x]+[y]=\{x'+y': x\in[x], y\in[y]\}$. That is $[x]+[y]$ is the set sum, so the hypersum $[x]\boxplus [y]=\{[z]:z\in [x]+[y]\}$, so after taking the set sum, we take all coset representatives.

\begin{remark}\label{rem:rem1}
We note that if $z\in [x_1]+[x_2]+...+[x_n]$, then $[z]\in [x_1]\boxplus[x_2]\boxplus...\boxplus[x_n]$.
\end{remark}

We know that these operations are well-defined and that this quotient forms a hyperfield by \cite{krasner1983class}. When we are working with a quotient hyperfield $K=F/G$, we make the convention to denote the elements of $K$ using the equivalence class notation $[x]\in K$, so we can distinguish between the element $[x]\in K$ and the element $x\in F$.

\begin{example}

(Fields) Fields form a quotient hyperfield with the quotient structure \\$F=F/\{1\}$. We see this as $x\sim y$ if and only if $x=y\cdot 1=y$. Thus, we have that $[x]=\{x\}$ for all $x\in F$. 

\end{example}

\begin{example}

(Krasner Hyperfield) The Krassner hyperfield can be obtained as a quotient of any field $F$ by quotienting out by the entire multiplicative subgroup, that is $\KK= F/F^\times$. We can see this as our two equivalence classes will be $[0],[1]$ where $[0]$ corresponds to the additive identity in $F$, and $[1]$ corresponds to all of $F^\times$ that is $[1]$ corresponds to any non-zero element.

\end{example}

\begin{example}

(Hyperfield of Signs) The hyperfield of signs can be obtained as a quotient of $\RR$ by $\RR_{>0}$ (or in general any ordered field $(F,<)$ as $F/F_{>0}$). When we look at the equivalence classes of $\SS=\RR/\RR_{>0}$, we have $[1]=\{x:x>0\}$, $[-1]=\{x:x<0\}$, and $[0]=\{0\}$ we have the three elements we saw corresponding to the positive numbers, the negative numbers, and zero. Thus, we now see why it made sense to interpret $[1]$ as any positive number, $[-1]$ as any negative number, and $[0]$ as zero.

\end{example}

\begin{example}\label{ex:ex1}

(Weak Hyperfield of Signs) The weak hyperfield of signs admits a quotient structure $\WW=\FF_p/(\FF_p^\times)^2$ where $p$ is a prime such that $p\geq 7$ and $p\equiv 3\pmod 4$. The fact that this quotient structure works for different primes will be key to showing that we can find two non-isomorphic minimal extensions of a hyperfield containing a root to some polynomial. We will see that taking different quotient structures can lead to different minimal extensions.

\end{example}

\begin{example}

(Phase Hyperfield) The phase hyperfield is a hyperfield with quotient structure $\PP=\CC/\RR_{>0}$. Since we are modding out by the reals, we can think of the phase hyperfield as $S^1\cup \{0\}$ where $S^1$ is the unit circle in the complex plane. We should think of the operations geometrically. Thus, multiplication is the usual multiplication over $\CC$ that is given two points $[x],[y]\in S^1$, $[x]\odot [y]$ is the point on $S^1$ given by adding the angles of $[x]$ and $[y]$. Then addition is given by $[x]\boxplus [0] = \{[x]\}$. For inverses we have $[x]\boxplus[-x]=\{[0],[x],[-x]\}$, we should think of this as since the line through $[x]$ and $[-x]$ goes through the origin the only possible phases you can get are $[x]$ and $[-x]$, and we include $[0]$ as the line goes through the origin. Then for $[x],[y]\neq [0]$ and $[y]\neq[-x]$ we have that $[x]\boxplus[y]$ is the shorter of the two arcs between the points $[x]$ and $[y]$ on $S^1$.

\end{example}

\subsection{Strong and Weak Hyperfield Extensions} 

Since addition is multivalued, there are two natural notions of hyperfield extension that one can take. We shall call them \textit{weak hyperfield extensions} and \textit{strong hyperfield extensions}. The two definitions come from requiring either inclusion or equality. We shall define the hyperfield extensions in terms of homomorphisms of hyperfields. Let $K$ and $L$ be hyperfields and denote $\boxplus_K$ the hyperaddition in $K$ and $\boxplus_L$ the hyperaddition in $L$. Similarly, $\odot_K$ and $\odot_L$ are the multiplications in $K$ and $L$.

\begin{defn}
A \textit{weak hyperfield homomorphism} is a function $\phi:K\rightarrow L$ that satisfies the following:

$\phi(0)=0$

$\phi(1)=1$

$\phi(x\boxplus_K y)\subseteq \phi(x)\boxplus_L\phi(y)$

$\phi(x\odot_K y)=\phi(x)\odot_L \phi(y)$

\end{defn}

\begin{defn}

We say that $\phi: K\rightarrow L$ is a \textit{strong hyperfield homomorphsim} if in addition to being a weak hyperfield homomorphism we have that $\phi(x\boxplus_K y)= \phi(x)\boxplus_L\phi(y)$

\end{defn}

\begin{remark}
We remark that unlike classical field theory, homomorphisms need not be injective. For example, it is not to hard to show that for any hyperfield $K$ that $\phi:K\rightarrow \KK$, given by $\phi(0)=[0]$ and $\phi(x)=[1]$ for $x\neq 0$ is always a weak hyperfield homomorphism.
\end{remark}

Now we shall define hyperfield extensions in terms of homomorpisms of hyperfields; however, we shall require that these homomorphisms be injective. Since we have two notions of homomorphism, we will define two types of extensions.

\begin{defn}
For hyperfields $K$ and $L$, we say that $L$ is a \textit{weak hyperfield extension} of $K$ if there is an injective weak homomorphism of hyperfields $\phi: K\hookrightarrow L$. In this case, we say that $K$ is \textit{weak subhyperfield} of $L$. 

If $\phi$ is a strong homomorphism of hyperfields, we sat that $L$ is a \textit{strong hyperfield extension} of $K$ and we say that $K$ is a \textit{strong subhyperfield} of $L$
\end{defn}

The terminology of weak and strong is natural as if $\phi:K\rightarrow L$ is a strong hyperfield homomorphism, then it is also a weak hyperfield homomorphism. Similarly, strong hyperfield extensions are also weak hyperfield extensions. Ideally, we would like to find strong hyperfield extensions that will contain roots to polynomials. We shall do this for quotient hyperfields. Now let us give some examples of hyperfield extensions.

\begin{example}
The weak hyperfield of signs $\WW$ is a weak extension of the hyperfield of signs $\SS$. We see this via the homomorphism $\phi:\SS\hookrightarrow\WW$ given by $\phi([x])=[x]$. It is a simple exercise to see that this is a weak extension of the hyperfield of signs, but is not a strong extension.
\end{example}

We note, as in the above example, that for weak hyperfield extensions, we can have the size of the hyperfields be the same, but changing the way addition is defined can create extensions. We now give an example of a strong extension.

\begin{example}

The phase hyperfield $\PP=\CC/\RR_{>0}$ is a strong hyperfield extension of the hyperfield of signs $\SS=\RR/\RR_{>0}$. We see this via the homomorphism $\phi:\SS\hookrightarrow\PP$ given by $\phi([x])=[x]$. This is indeed the case, since the equivalence classes $[0]=\{0\}$, \\$[1]=\{x\in\RR: x>0\}$, $[-1]=\{x\in\RR:x<0\}$ are the same in both $\SS$ and $\PP$. Since the equivalence classes are the same, the operation of $\boxplus_\PP$ and $\odot_\PP$ is the same as in $\SS$. Thus, we have a strong extension.

\end{example}

We note that $\PP$ and $\SS$ have a similar structure as quotient hyperfields. Namely, $\PP$ and $\SS$ have the form $F/G$ and $L/G$ for $L$ a field extension of $F$. We can generalize this into the following lemma.

\begin{lemma}\label{lem:lem1}

For field extension $L/F$, for $G\leq F^\times$, we have that the quotient hyperfield $L/G$ is a strong hyperfield extension of $F/G$.

\end{lemma}

\begin{proof}

We shall construct a strong homomorphism of hyperfields. Let us define \\$\phi: F/G\rightarrow L/G$ given by $\phi([x])=[x]$. We shall first show that $[x]_{F/G}=[x]_{L/G}$ where $[x]_{F/G}$ is the equivalence class of $x$ in $F/G$, and $[x]_{L/G}$ is the equivalence class of $x$ in $L/G$. To see this, we see that $[x]_{F/G}=\{xg: g\in G\}$ and $[x]_{L/G}=\{xg: g\in G\}$, since the multiplicative subgroup $G$ is the same in both cases, we see that $[x]_{F/G}=[x]_{L/G}$. We shall now drop the subscript of $[x]$. Since the equivalence classes are the same in $F/G$ and $L/G$ we have that the map $\phi$ is well-defined and injective. Furthermore, we have that the strong homomorphism property holds since
\[
\phi([x]\odot_{F/G}[y])=\phi([xy])=[xy]=[x]\odot_{L/G}[y]=\phi(x)\odot_{L/G}\phi(y)
\]
and 
\[
\phi([x]\boxplus_{F/G}[y])=\phi(\{[z]:z\in [x]+[y]\})=\{[z]:z\in [x]+[y]\}=[x]\boxplus_{L/G}[y]=\phi([x])\boxplus_{L/G}\phi([y])
\]
Thus, we have that $L/G$ is a strong hyperfield extension of $F/G$.
\end{proof}

\subsection{Polynomials over Hyperfields}

Now there are several issues that we encounter when moving to polynomials over hyperfields. In field theory, polynomials over fields form rings; however, we will see that polynomials over hyperfields do not form hyperrings. We will see that the multiplication of polynomials over hyperfields is actually a multivalued operation. Thus, as we are not working over a hyperring, so we can't simply mimic the construction of field extensions as we do not have ideals generated by polynomials over hyperfields. 

\begin{defn}
A \textit{polynomial over a hyperfield} $K$ in variable $T$ is a formal sum
\[
k(T)=\op_{i=0}^n a_iT^i
\]
where each $a_i\in K$. We denote the set of all polynomials over a hyperfield $K$ in variable $T$ by $K[T]$.

\end{defn}

Thus, we see that polynomials look exactly the same over hyperfield as fields with the main difference being that we are taking a hypersum; thus, evaluating a polynomial at an element of $K$ will result in a set rather than a single element. We note that there have been several studies focused more on the algebraic structure of polynomials over hyperfield for example see \cite{article} or \cite{doi:10.1080/09720529.2003.10697978}.

We will make the convention that we will use $x$ as our variable for polynomials over fields, and $T$ will be our variable when we work over hyperfields.

Now the next issue that we encounter encounter is defining roots to polynomials, as when we plug in an element into our polynomial rather than getting back an element, we get a set. So we simply require that $0$ is in the set $k(\Root)$ for $\Root$ to be a root of $k$. We note that a study on the multiplicities of roots to polynomials over hyperfields was developed in \cite{baker2018descartes}.

\begin{defn}
For a polynomial $k\in K[T]$, we say that $\Root$ is a \textit{root} of $k$ if $0\in k(\Root)$.
\end{defn}

We will show that for a quotient hyperfield $K=F/G$, if $k\in K[T]$ does not have any roots, then we can find a strong hyperfield extension $L$ of $K$ such that $k$ contains a root in $L$. One might try the classical construction of modding out by the ideal generated by the polynomial; however, the next example shows that multiplication of polynomials over hyperfields is a multivalued operation. Thus, trying to construct an extension by modding out by an ideal would result in an algebraic structure with a multi-valued multiplication.

\begin{example}

(Hyperfield of Signs) Consider the polynomials $p(T)=1\boxplus T$ and\\ $q(T)=-1\boxplus T$ over the hyperfield of signs. If we were to multiply these polynomials in the usual way we see that:
\[
    (1\boxplus T)(-1\boxplus T)=-1 \boxplus ((-1)T\boxplus T)\boxplus T^2 = -1\boxplus (1\boxplus -1) T\boxplus T^2
\]
That is $p(T)q(T)=\{-1\boxplus T^2, -1\boxplus -T\boxplus T^2, -1 \boxplus T\boxplus T^2\}$, so multiplication of polynomials over hyperfield is not a single valued operation. 

\end{example}

Since quotient hyperfields arise from fields, for a quotient hyperfield $K=F/G$ there is a natural way to get a polynomial $\overline{f}\in K[T]$ from a polynomial $f\in F[x]$.

\begin{defn}
For a quotient hyperfield $K=F/G$, for a polynomial $f\in F[x]$, we define the \textit{induced polynomial} $\overline{f}\in K[T]$ as follows:
\[
f(x)=\sum_{i=0}^na_ix^i\leadsto \overline{f}=\op_{i=0}^n[a_i]T^i
\]
\end{defn}

\section{Constructing Hyperfield Extensions}\label{sec:ext}
 
 \subsection{Strong Extensions for Quotient Hyperfields}
 
For this entire section, let $F$ be a field, $G\leq F^\times$, and let $K=F/G$ be the corresponding quotient hyperfield.

Let $k(T)\in K[T]$ be a polynomial containing no roots in $K$. In this section, we will construct a strong hyperfield extension to $K$ that contains some root $[\root]$ to $k$. 

We begin with a key lemma that will be used in our construction.

\begin{lemma}\label{lem:lem2}

If $f\in F[x]$ with $f(x)=\sum_{i=0}^na_ix^i$, then for $\gamma\in F$, if $f(\gamma)=\alpha$, then for the induced polynomial $\overline{f}$, we have that $[\alpha]\in \overline{f}([\gamma])$.

\end{lemma}

\begin{proof}

We want to show that $[\alpha]\in \overline{f}([\gamma])$, now
\[
\overline{f}([\gamma])=\op_{i=0}^n[a_i][\gamma]^i=\op_{i=0}^n[a_i\gamma^i]
\]

Since $\alpha=\sum_{i=0}^n a_i\gamma^i$, we have that $\alpha\in [a_0\gamma^0]+[a_1\gamma^1]+...+[a_n\gamma^n]$, so by Remark \ref{rem:rem1}, we have that $[\alpha]\in\op_{i=0}^n[a_i\gamma^i]$ that is $[\alpha]\in\overline{f}([\gamma])$.
\end{proof}

\begin{corollary}\label{cor:cor1}

If $f\in F[x]$ has a root, then $\overline{f}\in K[T]$ has a root.

\end{corollary}

\begin{proof}
Since $f$ has a root, we have that there is some $\root$ such that $f(\root)=0$, so by Lemma \ref{lem:lem2} above $[0]\in \overline{f}([\root])$.
\end{proof}

\begin{lemma}\label{lem:lem3}

If $k\in K[T]$ contains no roots, then there exists a polynomial $f\in F[x]$ with  no roots in $F$ such that $\overline{f}=k$

\end{lemma}

\begin{proof}

Let $k(T)=\boxplus_{i=0}^n[a_i]T^i$ be a polynomial over $K$, then consider the polynomial $f(x)=\sum_{i=0}^na_ix^i$. Clearly, $\overline{f}=k$, now assume for sake of contradiction that $f$ has some root $\root$ in $F$, but by the above corollary, this would imply that $\overline{f}=k$ would have a root, but this contradicts the assumption that $k$ does not have a root.
\end{proof}

 The idea behind our construction will be to find a polynomial $f\in F[x]$ with no roots such that $\overline{f}=k$. Then as $f$ will have is over the field $F$, we can apply classical field theory to find an extension $F'/F$ that contains a root to $f$, then we will form the quotient $F'/G$ and show that this is a strong hyperfield extension of $K$.

\begin{theorem}\label{thm:thm1}

If $k\in K[T]$ contains, no root, then there is a strong hyperfield extensions $K\subseteq L$.

\end{theorem}

\begin{proof}

Let us construct the hyperfield extension. Since $k(T)$ contains no roots, by Lemma \ref{lem:lem3} there is a polynomial $f\in F[x]$ that contains no roots over $F$ such that $\overline{f}=k$. Since $f\in F[x]$ contains no roots over $F[x]$, there is a field extension $F\subseteq F'$ such that $f$ contains a root $\root\in F'$. Now let $L=F'/G$ be the quotient hyperfield taken by modding out by the same multiplicative subgroup by Lemma \ref{lem:lem1} we have that $L$ is a strong hyperfield extension of $K$. Now since $f$ has a root in $F'$ by Corollary \ref{cor:cor1} we have that $\overline{f}=k$ has a root in $L$. Thus, $L$ is a strong hyperfield extension of $K$ that contains a root to $k$.
\end{proof}

\begin{example}

The polynomial $1\boxplus T^2$ contains no roots over $\SS$. Since $\SS$ has the quotient structure of $\RR/\RR_{>0}$, we lift $1\boxplus T^2$ to $1+x^2\in \RR[x]$. Now we take the field extension containing a root to $1+x^2$, namely $\CC$. Thus, our extension is $\CC/\RR_{>0}=\PP$.

\end{example}

The construction that we gave relies on the fact that we could lift our polynomial over a hyperfield to a polynomial with no roots over a field. Thus, we do not have a construction for extensions for non-quotient hyperfields. 

\begin{question}
Given a polynomial over a hyperfield is there a hyperfield extensions (weak or strong) that contains a root to the polynomial?
\end{question}

\section{An Example of Two Non-Isomorphic Minimal Extensions}\label{sec:minext}

In this section we will give an explicit example of a polynomial over a hyperfield such that there are two non-isomorphic minimal extensions containing a root to this polynomial. To do this, we will show that a certain extension is minimal, and show that the other extension does not contain the minimal extension as a subhyperfield (in both the weak and strong sense).

Now for our construction, we will consider the weak hyperfield $\WW$. As we saw in Example \ref{ex:ex1}, we know that if $p$ is a prime such that $p\geq 7$, and $p\equiv 3\pmod 4$, then $\WW$ has the quotient structure $\FF_p/(\FF^\times_p)^2$ where $(\FF^\times_p)^2$ denotes the non-zero squares. 

The polynomial $k(T)=1\boxplus T^2$ contains no roots over $\WW$. Thus, by taking the two different quotient structures of $\WW=\FF_7/(\FF_7^\times)^2$ and $\WW=\FF_{11}/(\FF_{11}^\times)^2$, we get strong extensions containing roots using the construction from Theorem \ref{thm:thm1}. The extensions have the quotient structures $\FF_{49}/(\FF_7^\times)^2$ and $\FF_{121}/(\FF_{11}^\times)^2$ respectively. 

We will first show that $\FF_{49}/(\FF_7^\times)^2$ is a minimal extension, then we shall show that $\FF_{121}/(\FF_{11}^\times)^2$ does not contain $\FF_{49}/(\FF_7^\times)^2$ as a subhyperfield, so any the minimal subhyperfield of $\FF_{121}/(\FF_{11}^\times)^2$ containing a root to $k(T)$ is not isomorphic to $\FF_{49}/(\FF_7^\times)^2$; thus, there are two non-isomorphic minimal extensions.

\begin{theorem}\label{lem:lem4}

The extension $\WW\subseteq \FF_{49}/(\FF_7^\times)^2$ is a minimal extension of $\WW$ containing a root to $k(T)$.

\end{theorem}

\begin{proof}

Let us denote $L=\FF_{49}/(\FF_7^\times)^2$. We first note that by construction $L$ contains a root to $k(T)$, so we simply need to show that this extension is indeed minimal. It will be helpful to think of $\FF_{49}$ as $\FF_7[i]$ where $i^2=-1$. Furthermore, we note that $(\FF_{7}^\times)^2=\{1,2,-3\}$. Now to see that this extension is minimal, let us consider the equivalence classes:

{
$[0]=\{0\}$

$[1]=\{1,2,-3\}$

$[-1]=\{-1,-2,3\}$

$[i]=\{i,2i,-3i\}$

$[-i]=\{-i,-2i,3i\}$

$[i+1]=\{i+1,2i+2,-3i-3\}$

$[-i+1]=\{-i+1,-2i+2,3i-3\}$

$[i+2]=\{i+2,2i-3,-3i+1\}$

$[-i+2]=\{-i+2,-2i-3,3i+1\}$

$[i+3]=\{i+3,2i-1,-3i-2\}$

$[-i+3]=\{-i+3,-2i-1,3i-2\}$

$[i-1]=\{i-1,2i-2,-3i+3\}$

$[-i-1]=\{-i-1,-2i-2,3i+3\}$

$[i-2]=\{i-2,2i+3,-3i-1\}$

$[-i-2]=\{-i-2,-2i+3,3i-1\}$

$[i-3]=\{i-3,2i+1,-3i+2\}$

$[-i-3]=\{-i-3,-2i+1,3i+2\}$
}

We note that the elements of $L^\times$ form a cyclic group of order $16$. We know this as $((\FF_{49})^\times, \cdot)$ is a cyclic group, and $L^\times$ is a quotient group. Hence, it is cyclic. 

Now let $K$ be a weak subhyperfield of $L$ such that $\WW\subseteq K\subseteq \FF_{49}/(\FF_7^\times)^2$, and such that $K$ contains a root to $k(T)$. Our goal will be to show that $K^\times=L^\times$, that is, $K$ contains all nonzero elements of $L$. This will prove that $K=L$, and $L$ is a minimal extension containing a root to $k(T)$. We note that as $K\subseteq L$, we have that $K^\times \leq L^\times$; thus, by Lagrange's theorem $\vert K^\times \vert \mid 16$.

Since $\WW\subseteq K$, we know that $[0],[1],[-1]\in K$. Furthermore, since $K$ contains a root to k(T) and $K\subseteq L$ we know that either $[i]$ or $[-i]$ is an element of $K$ as these are the roots of $k(T)$ in $L$. However, we have that both of these elements are in $K$ as $[-i]=[-1]\odot [i]$ and as $K$ is a hyperfield, it is closed under $\odot$.

Now since $L$ is a weak hyperfield extension of $K$ and $[1],[i]\in K$, we have that:
\[
[1]\boxplus_K [i]\subseteq [1]\boxplus_L [i]=\{[i+1],[i+2],[i-3]\}
\]
Thus, $K$ contains some element $y\in\{[i+1],[i+2],[i-3]\}$, so $[1],[-1],[i],[-i],y\in K$. That is $K$ contains at least $5$ nonzero elements. Since $\vert K^\times \vert\mid 16$, $\vert K\vert =8$ or $\vert K\vert = 16$.

Assume for sake of contradiction that $K\neq L$, then $\vert K^\times \vert=8$, so $K^\times$ will contain all elements of order at most $8$. Thus, we have would have that $K^\times=\{[1],[-1],[i],[-i],[1+i],$\\$ [i-1],[-i+1],[-i-1]\}$ as these are the only elements of order at most $8$ in $L^\times$.

However, since $L$ is a weak hyperfield extension of $K$, we have that:
\[
[i]\boxplus_K[i+1]\subseteq [i]\boxplus_L[i+1]=\{[i-3],[-i-3],[-i+2]\}
\]
However, none of these elements are in $K$; thus, $K$ cannot be a hyperfield a contradiction. Hence, we have that $K=L$. Namely, $L$ is a minimal extension of $\WW$ containing a root to $k(T)$. 
\end{proof}

Now that we have shown that $\FF_{49}/(\FF_7^\times)^2$ is a minimal extension. We shall show that $\FF_{49}/(\FF_7^\times)^2$ is not contained as a subhyperfield of $\FF_{121}/(\FF_{11}^\times)^2$.

\begin{theorem}
$\FF_{49}/(\FF_7^\times)^2$ is not contained as a subhyperfield of $\FF_{121}/(\FF_{11}^\times)^2$. Thus, minimal extensions containing roots are not unique for polynomials over hyperfields.
\end{theorem}

\begin{proof}
Assume for sake of contradiction that $ \FF_{121}/(\FF_{11}^\times)^2$ is a hyperfield extension of $\FF_{49}/(\FF_7^\times)^2$. Then  $(\FF_{49}/(\FF_7^\times)^2)^\times \leq (\FF_{121}/(\FF_{11}^\times)^2)^\times$, so $16\mid \vert(\FF_{121}/(\FF_{11}^\times)^2)^\times\vert=24$ a contradiction. Thus, $\FF_{49}/(\FF_7^\times)^2$ is not a subhyperfield of $\FF_{121}/(\FF_{11}^\times)^2$. 

Namely, the minimal subhyperfield of $\FF_{121}/(\FF_{11}^\times)^2$ containing a root to $k(T)$ is not isomorphic to the minimal extension $\FF_{49}/(\FF_7^\times)^2$, that is, minimal extensions containing roots need not be unique for polynomials over hyperfields. 
\end{proof}

\bibliographystyle{alpha}
\bibliography{references}

\end{document}